\documentclass[11pt]{amsart}
\usepackage{amssymb,amsmath,amsthm,amsfonts,mathrsfs}
\usepackage[hmargin=3cm,vmargin=3.5cm]{geometry}
\usepackage[dvipsnames,table,xcdraw]{xcolor}
\usepackage[colorlinks = true,
            linkcolor = Fuchsia,
            urlcolor  = ForestGreen,
            citecolor = WildStrawberry,
            anchorcolor = blue]{hyperref}

\usepackage{stackengine}

\usepackage[all,2cell]{xy} \UseAllTwocells 
\usepackage{tikz-cd}

\usepackage{color}
\usepackage{graphicx,psfrag}
\usepackage{amscd}  
\usepackage{stmaryrd}  
\usepackage[all,2cell]{xy} \UseAllTwocells \SilentMatrices
\usepackage{tikz}
\usepackage[dvips]{epsfig}
\usepackage{comment}
\usepackage{kbordermatrix}
\usepackage{cancel,scalefnt}

\usepackage{array}
\newcolumntype{L}{>{$}l<{$}} 

\theoremstyle{definition}
\newtheorem{theorem}{Theorem}[section]
\newtheorem{lemma}[theorem]{Lemma}
\newtheorem{remark}[theorem]{Remark}

\newtheorem{proposition}[theorem]{Proposition}

\newtheorem{corollary}[theorem]{Corollary}

\newtheorem{example}[theorem]{Example}

\catcode`~=11 
\newcommand{\urltilde}{\kern -.15em\lower .7ex\hbox{~}\kern .04em}  
\catcode`~=13 

\usetikzlibrary{arrows,automata}
\usetikzlibrary{decorations.markings}
\usetikzlibrary{decorations.pathmorphing,shapes,decorations.text,shapes.geometric}
\usepackage{multirow}

\newcounter{sarrow}

\usepackage[colorinlistoftodos]{todonotes}

\definecolor{darkred}{rgb}{0.7,0,0} 
\newcommand{\defn}[1]{{\color{darkred}\emph{#1}}} 

\title[Nilpotent representations over equioriented cyclic quivers]{Nilpotent representations over equioriented cyclic quivers}

\author[Cornell Holmes]{Cornell Holmes}
\address{Department of Mathematics, Johns Hopkins University, Baltimore, MD 21218, USA}
\email{\href{mailto:cholme21@jhu.edu}{cholme21@jhu.edu}}

\author[Mee Seong Im]{Mee Seong Im}
\address{Department of Mathematics, Johns Hopkins University, Baltimore, MD 21218, USA}
\email{\href{mailto:meeseong@jhu.edu}{meeseong@jhu.edu}}

\makeatletter 
\@namedef{subjclassname@2020}{%
  \textup{2020} Mathematics Subject Classification}

\subjclass[2020]{Primary: 17B08, 20F18, 20F19, 20D15;
Secondary: 05C10, 05C20, 05C85.}
\date{August 17, 2026}

\providecommand{\keywords}[1]{\textbf{\textit{Key words and phrases.}} #1}

\keywords{Nilpotent cone, nilpotent endomorphisms, finite field, cyclic quiver, nilpotent representations, Boolean semiring, q-binomial coefficients.}

\begin{document}

\def\mfb{\mathfrak{b}}

\def\c{\mathsf{c}}
\def\B{\mathbb{B}}
\def\E{\mathsf E}
\def\F{\mathbb{F}}
\def\I{\mathsf I}
\def\R{\mathbb R}
\def\Q{\mathbb Q}
\def\Z{\mathbb Z}
\def\N{\mathbb N}
\def\C{\mathbb C}
\def\S{\mathbb S}
\def\V{\mathbf{V}}
\def\Set{\mathsf{Set}}
\def\FinSet{\mathsf{FinSet}}
\def\Prob{\mathsf{Prob}}
\def\Image{\mathsf{Im}}
\def\Lin{\mathsf{Lin}}
\def\Nil{\mathsf{Nil}}
\def\SS{\mathbb S}
\def\GL{\mathsf{GL}}
\def\Graph{\mathsf{Graph}}

\def\for{\mathsf{for}}
\def\Hom{\mathsf{Hom}}
\def\End{\mathsf{End}}
\def\rank{\mathrm{rank}}
\def\length{\mathsf{length}}

\def\mcN{\mathcal{N}}

\def\Der{\mathsf{Der}}
\def\Pol{\mathsf{Pol}}

\newcommand{\dmod}{\mathsf{-mod}}
\newcommand{\comp}{\mathrm{comp}} 
\newcommand{\col}{\mathrm{col}}
\newcommand{\adm}{\mathrm{adm}}  
\newcommand{\Ob}{\mathrm{Ob}}
\newcommand{\Cob}{\mathsf{Cob}}
\newcommand{\UCob}{\mathsf{UCob}}
\newcommand{\COB}{\mathcal{COB}}
\newcommand{\ECob}{\mathsf{ECob}}
\newcommand{\id}{\mathsf{id}}
\newcommand{\undM}{\underline{M}}
\newcommand{\im}{\mathsf{im}\:}
\newcommand{\coker}{\mathsf{coker}}
\newcommand{\Aut}{\mathsf{Aut}}
\newcommand{\tripod}{\mathsf{Td}}
\newcommand{\BBC}{\mathbb{B}(\mathcal{C})}
\newcommand{\Pmod}{\mathrm{pmod}}
\newcommand{\gammaoneR}{\gamma_{1,R}}  
\newcommand{\gammaoneRbar}{\overline{\gamma}_{1,R}} 
\newcommand{\gammaoneRprime}
{\gamma'_{1,R}}
\newcommand{\gammaoneRbarprime}
{\overline{\gamma}'_{1,R}}
\newcommand{\qbinom}[3]{\genfrac{[}{]}{0pt}{}{#1}{#2}_{#3}}

\def\l{\lbrace}
\def\r{\rbrace}
\def\o{\otimes}
\def\lra{\longrightarrow}
\def\ed{\mathsf{ed}}
\def\Ext{\mathsf{Ext}}
\def\ker{\mathsf{ker}\:}
\def\Rep{\mathsf{Rep}}
\def\Vect{\mathsf{Vect}}
\def\Free{\mathsf{Free}}
\def\mf{\mathfrak}
\def\mcC{\mathcal{C}}
\def\mcS{\mathcal{S}}  
\def\mcQC{\mathcal{QC}}
\def\mcA{\mathcal{A}}
\def\mcF{\mathcal{F}}
\def\mcE{\mathcal{E}}
\def\Fr{\mathsf{Fr}}  
\def\ev{\mathsf{ev}}  
\def\DAG{\mathsf{DAG}}
\def\Src{\mathsf{Src}}
\def\Hgt{\mathsf{Hgt}}

\def\bbn{\mathbb{B}^n}
\def\ovb{\overline{b}}
\def\tr{{\sf tr}} 
\def\det{{\sf det }} 
\def\one{\mathbf{1}}   
\def\kk{\mathbf{k}}  
\def\gdim{\mathsf{gdim}}  
\def\rk{\mathsf{rk}}
\def\Lie{\mathsf{Lie}}
\def\IET{\mathsf{IET}}
\def\SAF{\mathsf{SAF}}

\newcommand{\indexw}{\R_{>0}} 

\newcommand{\brak}[1]{\ensuremath{\left\langle #1\right\rangle}}
\newcommand{\oplusop}[1]{{\mathop{\oplus}\limits_{#1}}}
\newcommand{\addfigure}{\vspace{0.1in} \begin{center} {\color{red} ADD FIGURE} \end{center} \vspace{0.1in} }
\newcommand{\add}[1]{\vspace{0.1in} \begin{center} {\color{red} ADD FIGURE #1} \end{center} \vspace{0.1in} }
\newcommand{\vspin}{\vspace{0.1in} }

\newcommand\circled[1]{\tikz[baseline=(char.base)]{\node[shape=circle,draw,inner sep=1pt] (char) {${#1}$};}} 

\let\oldemptyset\emptyset
\let\emptyset\varnothing

\let\oldtocsection=\tocsection
\let\oldtocsubsection=\tocsubsection
\renewcommand{\tocsection}[2]{\hspace{0em}\oldtocsection{#1}{#2}}
\renewcommand{\tocsubsection}[2]{\hspace{1em}\oldtocsubsection{#1}{#2}}

\renewcommand{\kbldelim}{(}
\renewcommand{\kbrdelim}{)}

\def\MK#1{{\color{red}[MK: #1]}}
\def\bfred#1{{\color{red}#1}}

\def\MSI#1{{\color{purple}[MSI: #1]}}
\def\bfred#1{{\color{purple}#1}}

\def\CL#1{{\color{magenta}[CL: #1]}}
\def\bfred#1{{\color{magenta}#1}}


\begin{abstract}
We give a geometric description of nilpotent representations of equioriented cyclic quivers over a field via an explicit bijection. Over a finite field, this yields formulas for the number of nilpotent representations and for the probability that a representation is nilpotent. We also give a refinement by rank. Over the Boolean semiring, we identify nilpotent semirepresentations with directed acyclic graphs, derive a recursive formula for their number, and determine the asymptotic decay rate of the probability of nilpotence.
\end{abstract}

\maketitle
\tableofcontents

%
%

\section{Introduction}
\label{sec_intro}

The nilpotent cone is predominant in representation theory, algebraic geometry, and mathematical physics. It is the set of all elements that are nilpotent in the adjoint representation. For a reductive algebraic group $G$ with Lie algebra $\mathfrak{g} = \Lie(G)$ over an algebraically closed field, it is the zero fiber of the adjoint map $\mathfrak{g}\twoheadrightarrow \mathfrak{g}/\!\!/G\simeq \mathfrak{t}/W$, where $\mathfrak{t}$ is the Lie algebra of the maximal torus $T\subseteq G$ and $W =N_G(T)/T$ is the Weyl group, with $N_G(T)$ being the normalizer of $T$ in $G$. The nilpotent cone plays a fundamental and central role in the Springer and Grothendieck--Springer resolutions~\cite{Spr76,CG97,Lus84,Spa82,MR3836769,Im_Scrimshaw_parabolic}, the geometry of singularities and Slodowy slices~\cite{GG02}, representations of quivers and loop Grassmannians~\cite{BB20,MV22}, and Higgs bundles as the zero fiber of the Hitchin map~\cite{Hit87,Hit87_stable_bundles,Lau88}, to name a few.

Let $V$ be an $n$-dimensional vector space over a field $\kk$, and let
$\mathcal N(V)$ denote the set of nilpotent endomorphisms of $V$. When $\kk=\mathbb F_q$, a finite field, Fine and Herstein proved in \cite[Theorem 1]{FH58} that $|\mathcal N(V)|=q^{n(n-1)}=|V|^{n-1}$. Equivalently, the probability that an endomorphism is nilpotent is $q^{-n}$.
Over an arbitrary field $\kk$, Leinster generalized this result by
constructing a bijection 
\begin{equation}\label{intro_Leinster_bijection}
    \mathcal{N}(V)\times V\cong\End_{\kk}(V)
\end{equation}
in \cite[Theorem 5]{Lei21}.

In \cite{CIKLR25, CILR25}, the authors extended
Leinster's bijection to pairs of linear maps. Let $V$ and $W$ be
finite-dimensional vector spaces over $\kk$, and consider the pair $(f,g)\in\Hom_\kk(V,W)\times\Hom_\kk(W,V)$ of maps. A pair $(f,g)$ is \defn{nilpotent} if the composition $g\circ f$ is nilpotent as an endomorphism of $V$. Write $\mathcal N(V,W)$ for the set of nilpotent pairs and $V\cup_0W$ for the union of $V$ and $W$ along their zero vector.
The authors gave a bijection
\begin{equation}\label{eqn_intro_pairs_bijection}
    \mathcal N(V,W)\times V\times W\cong\Hom_\kk(V,W)\times\Hom_\kk(W,V)\times(V\cup_0W).
\end{equation}
When $\kk=\mathbb F_q$, $\dim V = m$, and $\dim W = n$, it follows that the probability that a pair $(f,g)$ is nilpotent is $q^{-m}+q^{-n}-q^{-m-n}$.

Fix an integer $k\geq 2$ and let $Q$ be the cyclic quiver
\[
1 \to 2 \to \ldots \to k-1 \to k \to 1.
\]
In this paper, we extend~\eqref{eqn_intro_pairs_bijection} to linear representations of $Q$ over $\kk$.
Fix a dimension vector $\beta=(\beta_1,\ldots,\beta_k)\in\mathbb Z^k_{\geq 0}$ and vector spaces $V_1,\ldots,V_k$ over $\kk$ with $\dim V_i=\beta_i$ for each $1\leq i\leq k$. Throughout, subscripts are understood modulo $k$. The \defn{representation space} of a quiver $Q$ with dimension vector $\beta$ is $\Rep(Q,\beta)=\prod_{i=1}^k\Hom_\kk(V_i,V_{i+1})$.
A representation $(f_1,\ldots,f_k)\in\Rep(Q,\beta)$ is the cyclic analogue of a single endomorphism $g:V_1\to V_1$. Instead of iterated powers $g^\ell$, we consider cyclic compositions $f_{i+\ell-1}\cdots f_i$. Just as an endomorphism is nilpotent if iterated powers eventually vanish, a representation is \defn{nilpotent} if all sufficiently long compositions of consecutive maps vanish.

Write $\mathcal N(\beta)\subseteq\Rep(Q,\beta)$ for the subset of nilpotent representations, and $Z\subseteq\prod_{i=1}^kV_i$ for the set of tuples $(v_1,\ldots,v_k)$ for which some $v_i=0$. Our first main result is Theorem~\ref{thm_nilp_geometric_bijection}, where we construct a bijection
\begin{equation}\label{eqn_intro_main_result}
    \mathcal N(\beta)\times \prod_{i=1}^k V_i\cong\Rep(Q,\beta)\times Z.
\end{equation}
When $k=2$, we have $Z\cong(V_1\cup_0V_2)$. Thus,~\eqref{eqn_intro_main_result} recovers~\eqref{eqn_intro_pairs_bijection}.

Given $\beta,\,\gamma\in\mathbb Z_{\geq 0}^k$, denote the dot product $\beta\cdot\gamma:=\sum_{i=1}^k\beta_i\gamma_i$. Define the cyclic permutation
$\sigma = (12\cdots k)$ in the symmetric group $S_k$, and let $\sigma$ act on $\mathbb Z_{\geq 0}^k$ via
$\sigma \beta = (\beta_2,\ldots, \beta_k,\beta_1)$.
Writing $\mathbf{1} := (1,\ldots, 1)$ for the all-ones vector, we have
\[\beta \cdot \sigma \beta = \sum_{i=1}^k \beta_i \beta_{i+1}=\dim\Rep(Q,\beta),\quad \mathbf{1} \cdot \beta = \sum_{i=1}^k \beta_i=\dim\prod_{i=1}^kV_i.\]
When $\kk$ is the finite field $\mathbb F_q$, we use~\eqref{eqn_intro_main_result} to deduce the cardinality of $\mathcal N(\beta)$ in Theorem~\ref{thm_card_nilp_finite_field}:
\begin{equation}
\label{eqn_intro_cardinality_formula}
     N(\beta) =q^{\beta \cdot \sigma\beta
     - \mathbf{1} \cdot \beta }
    \left(q^{ \mathbf{1} \cdot \beta } -\prod_{i=1}^k(q^{\beta_i}-1)\right).
\end{equation}
Thus, the probability that a representation in $\Rep(Q,\beta)$ is nilpotent is
\begin{equation}
\label{intro_eqn_prob_nil_cyclic_rep}
\Prob(\mcN(\beta)) 
= 1 - \prod_{i=1}^k(1 - q^{-\beta_i}).
\end{equation}
See Proposition~\ref{prop_prob_nil_rep_fin_field} for more detail.

We can also view $\mathcal N(\beta)$ as the nullcone for the action of $\prod_{i=1}^k\GL(V_i)$ on $\Rep(Q,\beta)$ by the change of basis at each vertex. The motives of such nullcones over $\mathbb C$ are studied in~\cite{GR26}. For more on the enumeration of nilpotent representations over finite fields, see~\cite{BSV20}. In the set-theoretic context, eventually constant set-valued representations replace nilpotent representations over $\kk$. Eventually constant set-valued representations of the cyclic quiver are treated in~\cite{GHI26}, and more general quivers are studied in~\cite{GHI26_set_quiver}.

Next, we consider semirepresentations of $Q$ over the Boolean semiring $\mathbb B=\{0,1\}$, where $1+1=1$. Fix $\beta\in\mathbb Z_{\geq 0}^k$ and let $M_1,\ldots,M_k$ be free $\mathbb B$-semimodules of ranks $\beta_1,\ldots,\beta_k$, respectively.
As in the linear setting, the \defn{semirepresentation space} is
$\Rep_{\mathbb B}(Q,\beta):=\prod_{i=1}^k\Hom_{\mathbb B}(M_i,M_{i+1})$,
and a semirepresentation is \defn{nilpotent} if all sufficiently long compositions of consecutive maps vanish. Write $\mathcal N_{\mathbb B}(\beta)$ for the set of nilpotent semirepresentations and set $N_{\B}(\beta):=|\mathcal N_\mathbb B(\beta)|$.

We construct a directed graph $G_\beta$ and identify $\Rep_\mathbb B(Q,\beta)$ with its set of spanning subgraphs. This identification restricts to a bijection $\mathcal N_\mathbb B(\beta)\cong\DAG(G_\beta)$, where $\DAG(G_\beta)$ is the set of spanning directed acyclic subgraphs (DAGs) of $G_\beta$. Therefore,
$N_{\B}(\beta)=|\DAG(G_\beta)|$.
To determine $N_\B(\beta)=|\DAG(G_\beta)|$, we apply inclusion-exclusion over source vertices in $G_\beta$, which recursively expresses $N_{\B}(\beta)$ in terms of classes of spanning DAGs on certain induced subgraphs of $G_\beta$. The induced subgraphs correspond to graphs $G_\gamma$ for strictly smaller dimension vectors $\gamma<\beta$, where
$\gamma<\beta$ if $\gamma_i\leq\beta_i$ for each $1\leq i\leq k$ and $\gamma\neq\beta$. This gives $N_\B(\beta)$ in terms of smaller $N_\B(\gamma)=|\DAG(G_\gamma)|$.
Also, set
\[
\binom{\beta}{\gamma}:=\prod_{i=1}^k\binom{\beta_i}{\gamma_i}.
\]
In Theorem~\ref{thm_boolean_nilp_cardinality_recursion}, we prove that for every nonzero dimension vector $\beta$,
\[N_{\B}(\beta)=\sum_{\gamma<\beta}(-1)^{
\mathbf{1} \cdot (\beta-\gamma) - 1 }
\binom{\beta}{\gamma}2^{ (\beta-\gamma)\cdot \sigma \gamma } N_{\B}(\gamma).\]

In Section~\ref{subsection_rep_quivers}, we give a background on quiver representations over a field and the Boolean semiring. In Sections~\ref{subsection_dimension_vectors} and~\ref{subsection_rank_representations}, we establish notation for dimension vectors, the rank of quiver representations, and $q$-binomial coefficients.
The necessary graph constructions are given in Section~\ref{subsection_directed_walks_graphs}. We introduce nilpotent representations of the cyclic quiver over a field in Section~\ref{subsection_nilp_rep} and relate them to the representation space via an explicit bijection in Section~\ref{subsection_geom_descript_nil_maps}. We determine the cardinality of the nilpotent representations over a finite field in Section~\ref{subsection_enumeration_nil_map_finite_field}, and then enumerate the nilpotent representations by rank using $q$-binomial coefficients in Section~\ref{subsection_enumeration_nilpotent_maps_rank}. Finally, in Section~\ref{subsection_prob_cyclic_quiver_field} we give the probability that a representation of the cyclic quiver is nilpotent.

In Section~\ref{subsection_nilp_semirepresentations}, we derive a recursive formula for the cardinality of the nilpotent semirepresentations of a cyclic quiver over the Boolean semiring. We analyze the probability that a semirepresentation is nilpotent in Section~\ref{subsection_prob_Boolean}. In Section~\ref{section_comparing_set_theoretic_vs_eqns}, 
we relate the probability that a representation of the cyclic quiver over a finite field is nilpotent with the probability that a set-valued representation of the cyclic quiver is eventually constant, as given in \cite[Proposition 3.5]{GHI26}. We also compare the categories of the linear and set-valued representations of the cyclic quiver.

\section*{Acknowledgments}
The authors thank Mikhail Khovanov for productive discussions during the early stages of this paper. 
M.S. Im is grateful to the Institute for Advanced Study in Princeton, NJ for their hospitality and to Purdue University for the excellent working environment.
The authors were partially supported by Simons Collaboration Award 994328. C. Holmes was also partially supported by NSF grant DMS-2428878.

\section{Background}
\label{section_background}

\subsection{Quiver representations}
\label{subsection_rep_quivers}

A \defn{quiver} $Q=(Q_0,Q_1)$ consists of finite sets of \defn{vertices} $Q_0$ and \defn{arrows} $Q_1$. An arrow $\alpha\in Q_1$ is directed from its \defn{tail} $t(\alpha)\in Q_0$ to its \defn{head} $h(\alpha)\in Q_0$.
A \defn{path} $\alpha_{\ell-1}\cdots\alpha_0$ is a composable sequence of arrows.

A \defn{representation} of $Q$ with values in a category $\mathcal C$ consists of a collection of objects in $\mathcal C$ indexed by $Q_0$ and a collection of morphisms in $\mathcal C$ indexed by $Q_1$,
\[\{W_a :  a\in Q_0\},\quad \{W_\alpha:W_{t(\alpha)}\to W_{h(\alpha)} : \alpha\in Q_1\}.\]
When $\mathcal C$ has zero morphisms, a representation $W$ is \defn{nilpotent} if for all sufficiently long paths $\alpha_{\ell-1}\cdots\alpha_0$, the composition $W_{\alpha_{\ell-1}}\circ \cdots \circ W_{\alpha_0}$ is the zero morphism.

Alternatively, $Q$ determines a \defn{path category}: objects are vertices in $Q_0$ and morphisms are paths. We include the trivial path of length $0$ at each vertex $a\in Q_0$ to serve as the identity on $a$, and take composition as path concatenation. Then, representations of $Q$ can be viewed as functors from the path category to $\mathcal C$.

Throughout, we fix $k\geq 2$ and take all indices modulo $k$. We study the nilpotent representations of a fixed quiver $Q$ over two distinct categories.
Let $Q$ be the cyclic quiver with vertex set $Q_0=[k]=\{1,\ldots,k\}$ and arrows
\[
1 \xrightarrow{\alpha_1} 2 \xrightarrow{\alpha_2} \ldots \xrightarrow{\alpha_{k-2}}k-1\xrightarrow{\alpha_{k-1}}k\xrightarrow{\alpha_k}1.
\]

In Section~\ref{section_nilpot_rep_field} we consider the usual representations of $Q$ over a field $\kk$, those with values in the category of finite-dimensional $\kk$-vector spaces. Let $\beta=(\beta_1,\ldots,\beta_k)\in\mathbb Z_{\geq 0}^k$ be a \defn{dimension vector} of $Q$ and fix vector spaces $V_1,\ldots,V_k$ over $\kk$ of dimension $\dim V_i:=\beta_i$. Let $\Rep(Q,\beta)$ denote the set of representations $W$ that assign each vertex $i\in[k]$ to the vector space $V_i$. Each $W\in\Rep(Q,\beta)$ is determined by the $k$-tuple $(W_{\alpha_1},\ldots,W_{\alpha_k})\in\prod_{i=1}^k\Hom(V_i,V_{i+1})$, so we identify $\Rep(Q,\beta)$ with this product,
\[\Rep(Q,\beta):=\prod_{i=1}^k\Hom(V_i,V_{i+1}).\]

In Section~\ref{section_nilp_rep_Boolean} we consider the Boolean semiring $\mathbb B=\{0,1\}$, where $1+1=1$. We study \defn{semirepresentations} of $Q$ over $\mathbb B$, representations of $Q$ with values in the category of free $\mathbb B$-semimodules of finite rank. For each dimension vector $\beta\in\mathbb Z_{\geq 0}^k$, fix free $\mathbb B$-semimodules $M_1,\ldots,M_k$ with $\rank \: M_i=\beta_i$. As in the linear context, the set of semirepresentations with dimension vector $\beta$ identifies with the product
\[\Rep_{\mathbb B}(Q,\beta):=\prod_{i=1}^k\Hom_{\mathbb B}(M_i,M_{i+1}).\]

Concretely, we write (semi)representations $f$ in $\Rep(Q,\beta)$ or $\Rep_\mathbb B(Q,\beta)$ as
$f=(f_1,\ldots,f_k)$, with $f_i:V_i\to V_{i+1}$ or $f_i:M_i\to M_{i+1}$, respectively. In both contexts, $f$ is nilpotent if for each $1\leq i\leq k$ and sufficiently large $\ell\in \mathbb N$, the composition $f_{{i+\ell-1}}\circ\dots\circ f_{i}$ is zero. 
Note that $f$ is nilpotent if and only if $f_k\circ\cdots\circ f_1$ is nilpotent as an endomorphism of $V_1$ or $M_1$.
Define the subsets of nilpotent representations over $\kk$ and $\mathbb B$ as
\[
\mathcal N(\beta)\subseteq\Rep(Q,\beta),\quad \mathcal N_{\mathbb B}(\beta)\subseteq\Rep_\mathbb B(Q,\beta).
\]
Abbreviate the cardinalities as 
$N(\beta):= |\mathcal N(\beta)| $ and 
$N_{\B}(\beta):=|\mathcal N_{\mathbb B}(\beta)|$.

\subsection{Dimension vectors}
\label{subsection_dimension_vectors}

The operations in this section are used throughout for cardinality results.

Given $\beta,\,\gamma\in\mathbb Z_{\geq 0}^k$, denote the dot product $\beta\cdot\gamma:=\sum_{i=1}^k\beta_i\gamma_i$. Let 
$\sigma = (1\cdots k)$ be the cyclic permutation in the symmetric group $S_k$, and let $\sigma$ act on $\mathbb Z_{\geq 0}^k$ via
$\sigma \beta = (\beta_2,\ldots, \beta_k,\beta_1)$. Also, let $\mathbf 1:=(1,\ldots,1)\in\mathbb Z_{\geq 0}^k$. This gives
\[\beta\cdot\sigma\gamma=\sum_{i=1}^k\beta_i\gamma_{i+1},\quad \mathbf 1\cdot\beta=\sum_{i=1}^k\beta_i.\]

We order dimension vectors $\beta,\,\gamma\in\mathbb Z_{\geq 0}^k$ componentwise so that $\gamma\leq\beta$ if $\gamma_i\leq\beta_i$ for each $1\leq i\leq k$, and $\gamma<\beta$ if additionally $\gamma\neq\beta$. If $\gamma\leq\beta$, we extend the standard binomial coefficients $\binom{n}{r}$ by
\[\binom{\beta}{\gamma}:=\prod_{i=1}^k\binom{\beta_i}{\gamma_i}.\]

\subsection{Rank of representations and \texorpdfstring{$q$}{q}-binomial coefficients}
\label{subsection_rank_representations} 
We use these definitions in Section~\ref{subsection_enumeration_nilpotent_maps_rank}.
Consider representations of the cyclic quiver $Q$ over the finite field $\mathbb F_q$.

For each representation $f=(f_1,\ldots,f_k)\in\Rep(Q,\beta)$ and each $1\leq i<k$, the composition $f_i\cdots f_1$ is a linear map $V_1\to V_{i+1}$. So, $f$ induces the vector
\[\rank \: f:=(\beta_1,\rank (f_1),\rank (f_2 f_1),\ldots,\rank (f_{k-1}\cdots f_1))\in\mathbb Z^k_{\geq 0}.\]
Let $\mathcal R(\beta):=\{\rank \: f\in\mathbb Z^k_{\geq 0} : f\in\Rep(Q,\beta)\}$ and for each $\gamma\in\mathcal R(\beta)$, let \[\mathcal N_\gamma(\beta):=\{f\in\mathcal N(\beta) : \rank \: f=\gamma\}.\]
Observe that $\gamma\in\mathcal R(\beta)$ if and only if $\gamma_1=\beta_1$, $\gamma_{i+1}\leq\gamma_i$ for each $1\leq i<k$, and $\gamma\leq\beta$.

For integers $0\leq r\leq n$, the $q$-binomial coefficient is defined by
\[
\qbinom{n}{r}{q} := \frac{(1-q^n)(1-q^{n-1})\cdots (1-q^{n-r+1})}{(1-q^r)\cdots (1-q^2)(1-q)}.
\]
Equivalently, taking $[n]_q:=\frac{1-q^n}{1-q}$ and $[n]_q!:=[1]_q\cdots[n]_q$, we have
\begin{equation*}
    \qbinom{n}{r}{q} = \frac{[n]_q!}{[r]_q! [n-r]_q!}.
\end{equation*}
We define the $q$-binomial coefficient for $\gamma\leq\beta\in\mathbb Z_{\geq 0}^k$ in the same way as for the standard binomial coefficient,
\[\qbinom{\beta}{\gamma}{q}:=\prod_{i=1}^k\qbinom{\beta_i}{\gamma_i}{q}.
\]

We also have 
\begin{equation}
\label{eqn_GLr_Fq}
|\GL_r(\F_q)| = \prod_{i=0}^{r-1} (q^r - q^i) 
\end{equation}
since an $r\times r$ matrix is invertible if and only if its columns are linearly independent. We build a matrix in $\GL_r(\F_q)$ one column at a time, resulting in~\eqref{eqn_GLr_Fq}.

\subsection{Directed walks and graphs}
\label{subsection_directed_walks_graphs}
We use these definitions in  Section~\ref{section_nilp_rep_Boolean}.

A \defn{directed graph} $D$ consists of a \defn{vertex set} $V(D)$, and an (oriented) \defn{edge set} $E(D)\subseteq V(D)\times V(D)$. The edge $(u,v)\in E(D)$ has \defn{source} $u$ and \defn{target} $v$. We consider only finite directed graphs without parallel edges.

The \defn{union} of directed graphs $D$ and $D'$, $D\cup D'$, is given by 
\[V(D\cup D') := V(D)\cup V(D'),\quad E(D\cup D'):=E(D)\cup E(D').\]
The graph $D'$ is a \defn{subgraph} of $D$ if $V(D')\subseteq V(D)$ and $E(D')\subseteq E(D)$; it is a \defn{spanning subgraph} if $V(D')=V(D)$. For $U\subseteq V(D)$, let $D[U]$ denote the subgraph \defn{induced} by $U$, with vertex set $U$ and edge set $E(D)\cap (U\times U)$.

A \defn{directed walk} in $D$ of length $m\geq 1$ is a sequence of successive edges in $E(D)$,
\[(u_0,u_1),\,(u_1,u_2),\,\ldots,\,(u_{m-1},u_{m}).\]
The walk has \defn{source} $u_0$ and \defn{target} $u_{m}$. A \defn{directed acyclic graph} (DAG) is a directed graph $D$ with no closed walks, i.e., no walks with the same source and target. Since $D$ is finite, $D$ is acyclic if and only if its directed walks are bounded in length.
A vertex $u\in V(D)$ is a \defn{source} if it is not the target of any edge $e\in E(D)$. Let $\Src(D)\subseteq V(D)$ denote the set of source vertices of $D$.

Let $\DAG(D)$ be the set of spanning directed acyclic subgraphs of $D$.

\section{Nilpotent representations of cyclic quivers}
\label{section_nilpot_rep_field}

In Sections~\ref{subsection_geom_descript_nil_maps}, \ref{subsection_enumeration_nil_map_finite_field}, \ref{subsection_enumeration_nilpotent_maps_rank}, and \ref{subsection_prob_cyclic_quiver_field}, we assume that $Q$ is the equioriented cyclic quiver.

\subsection{Nilpotent representations}
\label{subsection_nilp_rep}
Let $V$ and $V'$ be finite-dimensional vector spaces with fixed ordered bases over a field $\kk$.

For each subspace $U\subseteq V$, the ordered basis of $V$ determines
an ordered basis of $U$ and a complementary subspace $U^{\c}$ such that $V=U\oplus U^{\c}$. We may take, for instance, the reduced row-echelon basis of $U$, and let $U^{\c}$ be the subspace spanned by the basis vectors of $V$ corresponding to the nonpivot columns.
    Applying~\cite[Theorem 5]{Lei21} with this data also produces a bijection
    \begin{equation}\label{eqn_Leinster_bijection}
        \mcN(U)\times U\cong\End(U),
    \end{equation}
where $\mcN(U)$ is the set of nilpotent endomorphisms of $U$.

Throughout, for each vector space with a fixed ordered basis, we implicitly fix a basis, a complement, and the bijection~\eqref{eqn_Leinster_bijection} for each of its subspaces.

\begin{lemma}\label{lmm_right_inverse_induced_bijection}
    Suppose $h:V\to V'$ and $g:V'\to V$ are linear maps and $U\subseteq V'$ is a subset. If $g$ is a right inverse of $h$, meaning $hg=\id_{V'}$, then $h$ and $g$ determine a bijection
    \[
    \varphi: h^{-1}(U)\xrightarrow\simeq\ker h\times U,\quad v\mapsto (v-gh(v),h(v)).
    \]
    In particular, we have $V\cong\ker h\times V'$.
\end{lemma}
\begin{proof}
    For each $v\in h^{-1}(U)$, $v-gh(v)\in\ker h$ since $h(v-gh(v))=h(v)-h(v)=0$. The inverse map is given by $\psi: (v,u)\mapsto v+g(u)$. 
    It is well-defined since for $(v,u)\in \ker h\times U$, 
    $h(v+g(u)) = h(v)+ hg(u) = u$.

    Now for $v\in h^{-1}(U)$, we have 
    $\psi\varphi(v) = \psi(v-gh(v), h(v)) = v-gh(v) +gh(v) = v$. Conversely, for $(v,u)\in \ker h\times U$, 
    we have $\varphi \psi(v,u) = 
    \varphi(v+g(u)) = (v+g(u) - gh(v+g(u) ), h(v+g(u))) = (v + g(u) -g(u), u) = (v,u) $ since $h(v)=0$.
\end{proof}

\begin{corollary}\label{cor_hom_evaluation_bijection}
    Each pair of elements $v\in V\setminus\{0\}$ and $w\in V'$ determines a bijection
    \[\Hom(V,V')\cong\{h\in\Hom(V,V') : h(v)=w\}\times V'.\]
\end{corollary}
\begin{proof}
    Let $\ev_v:\Hom(V,V')\to V'$ be the evaluation map, $f\mapsto f(v)$. The evaluation map admits a right inverse $g:V'\to\Hom(V,V')$.
    Choose a complementary subspace $\langle v\rangle^{\c}$ such that $\langle v\rangle\oplus \langle v\rangle^{\c} = V$.  Assign each vector $w'\in V'$ to the unique map $g(w'): V\to V'$ that takes $v$ to $w'$ and kills $\langle v\rangle^{\c}$. That is, 
    $g(w')(v) = w'$ and $g(w')|_{\langle v\rangle^{\c}} = 0$.

    By Lemma~\ref{lmm_right_inverse_induced_bijection}, the maps $\ev_v$ and $g$ yield a bijection
    \begin{equation}
        \Hom(V,V')\cong\{h\in\Hom(V,V') : h(v)=0\}\times V'.
    \end{equation}
    The result follows by translation, since
    \[
    \{h\in\Hom(V,V'):  h(v)=0\}+g(w)=\{h\in\Hom(V,V') : h(v)=w\}.\qedhere
    \]
\end{proof}

For each linear map $h\in\Hom(V,V')$, define the set
\[\mathcal N_h:=\{g\in\Hom(V',V): gh\in\mathcal N(V)\}.\]

\begin{lemma}\label{lmm_composite_nilpotent}
Suppose $h:V\to V'$ is a linear map. Then $h$ determines a bijection
\begin{equation}\label{eqn_composite_nilpotent}
    \mathcal N_h\times V\cong \ker h\times\Hom(V',V).
\end{equation}
\end{lemma}
\begin{proof}
Write $H:=\ker h$. Let $\pi:V\to V/H$ be the quotient map, and let $p:V'\to \im h$ be the projection that kills the complement $(\im h)^{\c}$.

We construct a linear map $\alpha$ and a right inverse $\beta$,
\[\alpha:\Hom(V',V)\to\End(V/H),\quad\beta:\End(V/H)\to \Hom(V',V).\]
For each $g\in \Hom(V',V)$, the composition $\pi gh\in\Hom(V,V/H)$ factors uniquely through $V/H$ since $H\subseteq\ker (\pi gh)$. Let $\alpha(g)\in\End(V/H)$ be the unique map such that
\begin{equation}\label{eqn_universal_property_quotient}
\alpha(g)\pi=\pi gh.
\end{equation}
The quotient map restricts to an isomorphism, $\overline\pi:H^{\c} \xrightarrow\simeq V/H$. Also, $ph:V\to\im h$ induces the isomorphism $\overline h:V/H\xrightarrow\simeq\im h$, where $\overline h$ is the unique map such that $\overline h \pi=ph$. For each $g'\in\End (V/H)$, let $\beta(g')\in\Hom(V',V)$ be the composition
\begin{equation}\label{eqn_beta_definition}
    V'\xrightarrow p\im h\xrightarrow{\overline h^{-1}}V/H\xrightarrow {g'}V/H\xrightarrow{{\overline\pi}^{-1}}H^{\c} \hookrightarrow V.
\end{equation}

To see that $\beta$ is a right inverse of $\alpha$, fix any $g'\in\End(V/H)$. Then $g'=\alpha(\beta(g'))$, since $g'$ satisfies~\eqref{eqn_universal_property_quotient} with respect to $\beta(g')$,
\[\pi \beta(g')  h=\pi{\overline\pi}^{-1}\circ g'\circ \overline h^{-1}ph=g'\pi.\]
Also, $\mathcal N_h=\alpha^{-1}(\mathcal N(V/H))$.
For each map $g\in\Hom(V',V)$,~\eqref{eqn_universal_property_quotient} implies $\alpha(g)^m\circ\pi=\pi\circ (gh)^m$ for each $m\in\mathbb N$. If $gh$ is nilpotent, then $\alpha(g)^m\circ\pi=0$ for some $m\in\mathbb N$. Thus $\alpha(g)^m=0$ since $\pi$ is surjective, so $\alpha(g)$ is nilpotent. Conversely, if $\alpha(g)$ is nilpotent, then $\pi\circ (gh)^m=0$ for some $m\in\mathbb N$, so \[\im((gh)^m)\subseteq\ker\pi=H=\ker h\subseteq\ker (gh).\]
Thus $(gh)^{m+1}=0$, and $gh$ is nilpotent.

By Lemma~\ref{lmm_right_inverse_induced_bijection}, $\alpha$ and $\beta$ induce bijections
\[\mathcal N_h\cong \ker\alpha\times\mathcal N(V/H),\quad \Hom(V',V)\cong \ker\alpha\times\End(V/H),\]
since $\alpha^{-1}(\mathcal N(V/H))=\mathcal N_h$ and $\alpha^{-1}(\End(V/H))=\Hom(V',V)$. Similarly, the quotient map $\pi:V\to V/H$ and its right inverse $({\overline\pi})^{-1}:V/H\to H^{\c} \hookrightarrow V$ yield a bijection $V\cong H\times V/H$. Finally, ${\overline\pi}:H^{\c} \to V/H$ pushes the ordered basis of $H^{\c}$ to an ordered basis of $V/H$, which specifies a bijection $\mathcal N(V/H)\times V/H\cong\End(V/H)$ by~\cite[Theorem 5]{Lei21}. Combining these bijections gives~\eqref{eqn_composite_nilpotent},
\begin{align*}
    \mathcal N_h\times V&\cong\big(\ker\alpha\times\mathcal N(V/H)\big)\times (V/H\times H)\\
    &\cong H\times\Big(\ker\alpha\times\big (\mathcal N(V/H)\times V/H\big)\Big)\\
    &\cong H\times\big(\ker\alpha\times\End(V/H)\big)\\
    &\cong H\times \Hom(V',V).\qedhere
\end{align*}
\end{proof}

\subsection{Geometric description of nilpotent maps}
\label{subsection_geom_descript_nil_maps}

Recall the notation in Section~\ref{subsection_rep_quivers}, and let $Q$ be the cyclic quiver. We fix a dimension vector $\beta\in\mathbb Z_{\geq 0}^k$ and vector spaces $V_1,\ldots,V_k$ over a field $\kk$ with $\dim V_i=\beta_i$ for each $1\leq i\leq k$. Let each vector space $V_i$ have an ordered basis.

We introduce two auxiliary sets $X$ and $Y$, and a set $Z$ that features in Theorem~\ref{thm_nilp_geometric_bijection}.
Let $X:=\prod_{i=1}^{k-1}\Hom(V_i,V_{i+1})$ and observe that $\Rep(Q,\beta)=X\times\Hom(V_k,V_1)$. Then writing
 $h$ for $(h_1,\ldots, h_{k-1})\in X$ and $v$ for $(v_1,\ldots, v_k) \in \prod_{i=1}^k V_i$, define
\[Y:=\{(h,v)\in X\times \prod_{i=1}^k V_i : v_1\in\ker\!(h_{k-1}\cdots h_1)\}.\]
Finally, let
$Z\subseteq \prod_{i=1}^k V_i$ be the subset of tuples $v=(v_1,\ldots,v_k)$ for which some $v_i=0$. Equivalently, $Z$ is the union
\[Z=\bigcup_{j=1}^k\left(\prod_{i<j}V_i\times\{0\}\times\prod_{l>j} V_l\right)\subseteq\prod_{i=1}^kV_i.\]

\begin{lemma}\label{lmm_X_Y_Z_bijection}
There is a bijection
      $X\times Z\cong Y$.
\end{lemma}
\begin{proof}
For each $v\in Z$, let $z(v):=\min\{i\in[k] : v_i=0\}$. Define a subset $Y_v\subseteq Y$ by
\[
Y_v := \{h\in X : h_i(v_i)=v_{i+1}\text{ for }i<z(v)\} 
\times
\{ v_1 \}\times\prod_{i=2}^{z(v)}V_i\times\{(v_{z(v)+1},\ldots,v_k)\}.
\]
This can be rewritten as 
\begin{equation*}
\begin{split}
Y_v &= \{h\in X : h_i(v_i)=v_{i+1}\text{ for }i<z(v)\} \times \\
& \hspace{5mm} 
\{ (u_1,\ldots, u_k):  u_1 = v_1, u_i \in V_i\mbox{ for }1 < i \leq z(v), u_i = v_i \mbox{ for } i > z(v)\}.
\end{split}
\end{equation*}
For every $(h,u)\in Y$, there exists a unique $v\in Z$ such that $(h,u)\in Y_v$. Take $w:=(u_1,h_1(u_1),\ldots,h_{k-1}\cdots h_1(u_1))$. Since $u_1\in \ker\!(h_{k-1}\cdots h_1)$, 
we see that the $k$-th component of $w$ is 
$w_k=0$. 
So $w\in Z$. Then $(h,u)\in Y_v$, where $v:=(w_1,\ldots,w_{z(w)},u_{z(w)+1},\ldots,u_k)$. For uniqueness, suppose $(h,u)\in Y_{v'}$. 
Then $v_1'=u_1 = w_1$. 
If $i<z(v')$, then 
$v_{i+1}' = h_i(v_i') = h_i(w_i) = w_{i+1}$.
By induction, $v'_i=w_i$ for $i\leq z(v')$. Hence, $z(v')=z(w)$, and $v'_i=u_i$ for $i>z(v')$, giving $v=v'$.

Since the subsets $X\times\{v\}$ partition $X\times Z$, and the subsets $Y_v$ partition $Y$, it suffices to construct a bijection $X\times\{v\}\cong Y_v$ for each $v\in Z$.

For $i < z(v)$, $v_i\neq 0$, so Corollary~\ref{cor_hom_evaluation_bijection} gives a bijection
\begin{equation}\label{eqn_hom_set_bijection}
    \Hom(V_i,V_{i+1})\cong\{h\in\Hom(V_i,V_{i+1}) : h(v_i)=v_{i+1}\}\times V_{i+1}.
\end{equation}
Since $X=\prod_{i=1}^{z(v)-1}\Hom(V_i,V_{i+1})\times\prod_{l=z(v)}^{k-1}\Hom(V_l,V_{l+1})$,
we see that~\eqref{eqn_hom_set_bijection} induces a bijection between $X\times\{v\}$ and
\begin{equation}\label{eqn_intermediate_slice}
\begin{aligned}
    \prod_{i=1}^{{z(v)}-1}&\Big(\{h_i\in\Hom(V_i,V_{i+1}) : h_i(v_i)=v_{i+1}\}\times V_{i+1}\Big)\times\prod_{l={z(v)}}^{k-1}\Hom(V_l,V_{l+1})\\
    &\cong\{h\in X : h_i(v_i)=v_{i+1}\text{ for }i<z(v)\}\times\prod_{i=2}^{z(v)}V_i.
\end{aligned}
\end{equation}
This set identifies with $Y_v$ after adding the fixed entries $\{v_1\}\times\{(v_{z(v)+1},\ldots,v_k)\}.$
\end{proof}

Recall that for a linear map $h\in\Hom(V_1,V_k)$,
\[\mathcal N_h=\{g\in\Hom(V_k,V_1): gh\in\mathcal N(V_1)\}.\]

\begin{theorem}
\label{thm_nilp_geometric_bijection}
There is a bijection
\begin{equation}
\label{eqn_bijection_nilp_field}
    \mathcal N(\beta)\times \prod_{i=1}^k V_i\cong\Rep(Q,\beta)\times Z.
\end{equation}
\end{theorem}
\begin{proof}
   A representation $(f_1,\ldots,f_{k})\in\Rep(Q,\beta)$ is nilpotent if and only if $f_{k}\cdots f_1\in\mathcal N(V_1)$, or equivalently, $f_k\in\mathcal N_{f_{k-1}\cdots f_1}$. Therefore, $\mathcal N(\beta)$ is isomorphic to the disjoint union,
\begin{equation}\label{eqn_Nilp_disjoint_union}
    \mathcal N(\beta)\cong \coprod_{h\in X}\mathcal N_{h_{k-1}\cdots h_1},\quad (f_1,\ldots,f_k)\mapsto ((f_1,\ldots,f_{k-1});f_k).
\end{equation}
For each $h\in X$, Lemma~\ref{lmm_composite_nilpotent} gives a bijection
\begin{equation}\label{eqn_nilp_kernel_bijection}
    \mathcal N_{h_{k-1}\cdots h_1}\times V_1\cong \ker\!(h_{k-1}\cdots h_1)\times\Hom(V_k,V_1).
\end{equation}
Thus, combining~\eqref{eqn_Nilp_disjoint_union} and~\eqref{eqn_nilp_kernel_bijection} yields a correspondence
\begin{equation}\label{eqn_N_V_1_bijection}
  \mathcal N(\beta)\times V_1\cong \coprod_{h\in X}\ker\!(h_{k-1}\cdots h_1)\times \Hom(V_k,V_1).
\end{equation}

Letting $Y_h = \ker\!(h_{k-1}\cdots h_1)\times\prod_{i=2}^kV_i$ for $h\in X$, 
we have 
$Y = \coprod_{h\in X} Y_h$, 
and 
$X\times\Hom(V_k,V_1)=\Rep(Q,\beta)$. Therefore,~\eqref{eqn_N_V_1_bijection} and Lemma~\ref{lmm_X_Y_Z_bijection} give
\begin{align*}
  \mathcal N(\beta)\times \prod_{i=1}^kV_i&\cong\left( \mathcal N(\beta)\times V_1\right)\times\prod_{i=2}^kV_i\\
  &\cong Y\times \Hom(V_k,V_1)\\
  &\cong X\times Z\times \Hom(V_k,V_1)\\
  &\cong\Rep(Q,\beta)\times Z.
\end{align*}
This proves the theorem.
\end{proof}

\subsection{Enumeration of nilpotent maps over finite fields}
\label{subsection_enumeration_nil_map_finite_field}
Let $\kk=\mathbb F_q$ be the finite field with $q\in\mathbb N$ elements, where $q$ is a prime power. Recall that the indices are taken modulo $k$.

\begin{theorem}
\label{thm_card_nilp_finite_field}
Let $Q$ be the cyclic quiver, 
and let $\beta$ be a dimension vector.
  We have 
    \begin{equation}
    \label{eqn_cardinality_formula}
        N(\beta) = q^{\sum_{i=1}^k \beta_i \beta_{i+1}-\sum_{i=1}^k \beta_i}
   \left(q^{\sum_{i=1}^k \beta_i}-\prod_{i=1}^k(q^{\beta_i}-1)\right).
    \end{equation}
\end{theorem}
    \begin{remark}
Equivalently, using the notation in Section~\ref{subsection_dimension_vectors} gives
\[ N(\beta) = q^{\beta\cdot\sigma\beta-\mathbf 1\cdot\beta}
   \left(q^{\mathbf 1\cdot\beta}-\prod_{i=1}^k(q^{\beta_i}-1)\right).\]
Expanding the product in~\eqref{eqn_cardinality_formula} also gives the alternative formula
\begin{equation}
\label{eqn_nilpot_k_tuple_cardinality}
\begin{split}
  N(\beta)  =  
 q^{\beta\cdot\sigma\beta -\mathbf 1\cdot\beta}
\left(\sum_{u=1}^{k} \sum_{1\leq i_1 < i_2 <\ldots < i_u \leq k} 
(-1)^{u-1} 
q^{\beta_1 + \ldots +\widehat{\beta}_{i_1} + \ldots + \widehat{\beta}_{i_u} + \ldots + \beta_k} \right),
\\ 
\end{split}
\end{equation}
where $\widehat{\beta}_{i_j}$ denotes removal of $\beta_{i_j}$ from the summation.
\end{remark}

We now prove Theorem~\ref{thm_card_nilp_finite_field}, which turns the bijection of Theorem~\ref{thm_nilp_geometric_bijection} into the enumeration formula in~\eqref{eqn_cardinality_formula}.

\begin{proof}
    Theorem~\ref{thm_nilp_geometric_bijection} gives
    \begin{equation}
         N(\beta) \cdot\left|\prod_{i=1}^k V_i\right|
    =\left|\Rep(Q,\beta)\right|\cdot|Z|.
    \end{equation}
    The products $\Rep(Q,\beta)= \prod_{i=1}^k\Hom(V_i,V_{i+1})$ and $\prod_{i=1}^kV_i$ have cardinality $q^{\sum_{i=1}^k \beta_i \beta_{i+1}}$ and $q^{\sum_{i=1}^k \beta_i}$, respectively.
    Also, $\prod_{i=1}^kV_i\setminus Z$ is the set of tuples $(v_1,\ldots,v_k)$ for which no $v_i=0$, which is $\prod_{i=1}^k\big(V_i\setminus\{0\}\big)$. Thus, $|Z|=q^{\sum_{i=1}^k \beta_i}-\prod_{i=1}^k(q^{\beta_i}-1)$,
    and the result follows.
\end{proof}

\subsection{Enumeration of nilpotent maps by rank}
\label{subsection_enumeration_nilpotent_maps_rank}
We continue to work over the finite field $\kk=\mathbb F_q$ and use the notation for $q$-binomial coefficients and ranks introduced in Section~\ref{subsection_rank_representations}.

We use the standard interpretation of $q$-binomial coefficients: for all $0\leq r\leq m$, there are $\qbinom{m}{r}{q}$ subspaces $U\subseteq \mathbb F_q^m$ of dimension $r$~\cite[Proposition 1.3.18]{Sta97}. Then for all $0\leq r\leq\min \{ m,n\}$, it follows as in~\cite[Lemma 4.7]{CILR25} that there are
\[\qbinom{m}{r}{q}\qbinom{n}{r}{q}\left|\GL_r(\mathbb F_q)\right|\]
maps $f\in\Hom(\mathbb F_q^m,\mathbb F_q^n)$ with $\rank \: f=r$. In particular, when $n\leq m$ there are $\qbinom{m}{n}{q}|\GL_n(\mathbb F_q)|$ surjections $f:\mathbb F_q^m\to\mathbb F_q^n$.

\begin{lemma}\label{prop_rank_subspace}
Let $U\subseteq\mathbb F_q^m$ and $U'\subseteq\mathbb F_q^n$ be subspaces with $a=\dim U\geq\dim U'=b$. Write $I(U,U')\subseteq\Hom(\mathbb F_q^m,\mathbb F_q^n)$ for the subset of
maps $f$ with $f(U)=U'$. Then
\[|I(U,U')|=q^{(m-a)n}\qbinom{a}{b}{q}|\GL_b(\mathbb F_q)|.\]
\end{lemma}
\begin{proof}
    Let $S(U,U')\subseteq\Hom(U,U')$ be the subset of surjective linear maps and view $\Hom(U,U')$ as a subset of $\Hom(U,\mathbb F_q^n)$. Fix a complementary subspace $U^{\c}$ with $U\oplus U^{\c} = \mathbb F_q^m$.
    
    Consider the restriction map
    \[\varphi:\Hom(\mathbb F_q^m,\mathbb F_q^n)\to\Hom(U,\mathbb F_q^n),\quad f\mapsto f|_U.\]
    Then $f(U)=U'$ if and only if $f|_U(U)=U'$, so
    $I(U,U')=\varphi^{-1}(S(U,U'))$. Also $\varphi$ is a surjective linear map and $\ker\varphi\cong\Hom(U^{\c}, \mathbb F_q^n)$. Therefore,
    \[|I(U,U')|=|\varphi^{-1}(S(U,U'))|=|\ker\varphi||S(U,U')|=q^{(m-a)n}\qbinom{a}{b}{q}|\GL_b(\mathbb F_q)|.\qedhere\]
\end{proof}

We return to $\Rep(Q,\beta)$ with $\beta\in\mathbb Z_{\geq 0}^k$ and corresponding $\mathbb F_q$-vector spaces $V_1,\ldots,V_k$.

\begin{theorem}\label{thm_N_gamma(beta)_cardinality}
    For each $\gamma\in\mathcal R(\beta)$,
     \[|\mathcal N_\gamma(\beta)|=q^{(\beta-\gamma)\cdot\sigma\beta+\gamma_k(\beta_1-1)}\qbinom{\beta}{\gamma}{q}\prod_{i=1}^{k-1}\left(\qbinom{\gamma_i}{\gamma_{i+1}}{q}|\GL_{\gamma_{i+1}}(\mathbb F_q)|\right).\]
\end{theorem}
\begin{proof}
Define the set
\[S_\gamma:=\{(U_1,\ldots,U_k) : \text{for each }1\leq i\leq k,\, U_i\subseteq V_i\text{ is a subspace with }\dim U_i=\gamma_i\}.\]
Note that $U_1=V_1$ for each $U\in S_\gamma$. Also, $|S_\gamma|=\prod_{i=1}^k\qbinom{\beta_i}{\gamma_i}{q}=\qbinom{\beta}{\gamma}{q}$.

A representation $f$ is in $\mathcal N_\gamma(\beta)$ if and only if $f_k\in\mathcal N_{f_{k-1}\cdots f_1}$ and
\[(V_1,\im f_1,\ldots,\im(f_{k-1}\cdots f_1))\in S_\gamma.\]
For each $U\in S_\gamma$, write $[U]\subseteq\mathcal N_\gamma(\beta)$ for the subset of $f$ with $\im(f_i\cdots f_1)=U_{i+1}$ for $1\leq i<k$. Then $f\in[U]$ if and only if $f_i(U_i)=U_{i+1}$ for each $1\leq i<k$ and $f_k\in\mathcal N_{f_{k-1}\cdots f_1}$. By Proposition~\ref{prop_rank_subspace}, there are
\begin{equation}\label{eqn_card_first_k-1_maps}
    \prod_{i=1}^{k-1}q^{(\beta_i-\gamma_i)\beta_{i+1}}\qbinom{\gamma_i}{\gamma_{i+1}}{q}|\GL_{\gamma_{i+1}}(\mathbb F_q)|=q^{\sum_{i=1}^{k-1}(\beta_i-\gamma_i)\beta_{i+1}}\prod_{i=1}^{k-1}\qbinom{\gamma_i}{\gamma_{i+1}}{q}|\GL_{\gamma_{i+1}}(\mathbb F_q)|
\end{equation}
such maps $f_1,\ldots,f_{k-1}$. Fixing $f_1,\ldots,f_{k-1}$,
Lemma~\ref{lmm_composite_nilpotent} gives a bijection
\begin{equation}\label{eqn_nilpotent_kernel}
    \mathcal N_{f_{k-1}\cdots f_1}\times V_1\cong \ker (f_{k-1}\cdots f_1)\times\Hom(V_k,V_1).
\end{equation}
By the rank-nullity theorem, $\dim(\ker(f_{k-1}\cdots f_1))=\beta_1-\gamma_k$. Thus,~\eqref{eqn_nilpotent_kernel} yields
\[|\mathcal N_{f_{k-1}\cdots f_1}|=\frac{q^{\beta_1-\gamma_k}q^{\beta_k\beta_1}}{q^{\beta_1}}=q^{\beta_k\beta_1-\gamma_k}.\]
Multiplying the number of choices for $f_1,\ldots,f_{k-1}$ given in~\eqref{eqn_card_first_k-1_maps} by the $q^{\beta_k\beta_1-\gamma_k}$ maps $f_k\in \mathcal N_{f_{k-1}\cdots f_1}$ gives
\begin{equation}\label{eqn_card_[U]}
    |[U]|=q^{\beta_k\beta_1-\gamma_k+\sum_{i=1}^{k-1}(\beta_i-\gamma_i)\beta_{i+1}}\prod_{i=1}^{k-1}\qbinom{\gamma_i}{\gamma_{i+1}}{q}|\GL_{\gamma_{i+1}}(\mathbb F_q)|.
\end{equation}

The exponent $\beta_k\beta_1-\gamma_k+\sum_{i=1}^{k-1}(\beta_i-\gamma_i)\beta_{i+1}$ is equal to $(\beta-\gamma)\cdot\sigma\beta+\gamma_k(\beta_1-1)$.
The sets $([U])_{U\in S_\gamma}$ partition $\mathcal N_\gamma(\beta)$ and each has the cardinality given in~\eqref{eqn_card_[U]}, so
\[|\mathcal N_\gamma(\beta)|=|S_\gamma||[U]|=q^{(\beta-\gamma)\cdot\sigma\beta+\gamma_k(\beta_1-1)}\qbinom{\beta}{\gamma}{q}\prod_{i=1}^{k-1}\left(\qbinom{\gamma_i}{\gamma_{i+1}}{q}|\GL_{\gamma_{i+1}}(\mathbb F_q)|\right).\qedhere\]
\end{proof}

\begin{corollary}\label{crl_N(beta)_cardinality_rank}
    We have
    \[
    N(\beta) = \sum_{\substack{\beta_1=\gamma_1\geq\gamma_2\geq \ldots\geq\gamma_k\geq 0\\\gamma\leq\beta}}q^{(\beta-\gamma)\cdot\sigma\beta+\gamma_k(\beta_1-1)}\qbinom{\beta}{\gamma}{q}\prod_{i=1}^{k-1}\left(\qbinom{\gamma_i}{\gamma_{i+1}}{q}|\GL_{\gamma_{i+1}}(\mathbb F_q)|\right).\]
\end{corollary}
\begin{proof}
    The sets $\mathcal N_\gamma(\beta)$ partition $\mathcal N(\beta)$ as $\gamma$ ranges over $\mathcal R(\beta)$, so
    \[N(\beta)=\sum_{\gamma\in\mathcal R(\beta)}|\mathcal N_\gamma(\beta)|.\]
    A dimension vector $\gamma$ is in $\mathcal R(\beta)$ precisely if $\gamma\leq\beta$ and $\beta_1=\gamma_1\geq\ldots\geq\gamma_k\geq 0$, so the result follows by Theorem~\ref{thm_N_gamma(beta)_cardinality}.
\end{proof}

\begin{example}
Suppose $k=3$, with $\beta=(\beta_1,\beta_2,\beta_3)\in\mathbb Z_{\geq 0}^3$. There is a bijection 
\begin{equation}
\begin{split}
V_1\times V_2\times V_3 \times \mcN(\beta) \simeq &\Hom(V_1,V_2)\times \Hom(V_2,V_3)\times \Hom(V_3,V_1)  \\ 
&\times ( (V_1 \times V_2\times\{0\})\cup (V_1\times\{0\}\times V_3)\cup (\{0\}\times V_2\times V_3) ),
\end{split}
\end{equation}
In the case when $\kk = \F_q$, 
we have 
\begin{equation}
\label{eqn_nilpotent_triple}
 N(\beta)  
= \sum_{r=0}^{\min\{\beta_1, \beta_2 \} }
\qbinom{\beta_1}{r}{q} 
\qbinom{\beta_2}{r}{q}
|\GL_r(\F_q)| 
q^{\beta_3(\beta_1 + \beta_2) - \beta_3 - r} (q^{\beta_3} + q^r -1),
\end{equation}
This simplifies as 
\begin{equation}
\label{eqn_nilp_triple_closed_form}
 N(\beta)  = q^{\beta_1 \beta_2 + \beta_2 \beta_3 + \beta_3 \beta_1 - (\beta_1 + \beta_2 + \beta_3)} (q^{\beta_1 + \beta_2} + q^{\beta_2 + \beta_3} + q^{\beta_3 + \beta_1} - q^{\beta_1} - q^{\beta_2} - q^{\beta_3} + 1). 
\end{equation}
The expression in parenthesis can also be written as $q^{\beta_1 + \beta_2 + \beta_3} - (q^{\beta_1} -1)(q^{\beta_2} - 1)(q^{\beta_3} - 1)$. It is the cardinality of
\[(V_1\times V_2\times\{0\})\cup(V_1\times\{0\}\times V_3)\cup(\{0\}\times V_2\times V_3)\subseteq V_1\times V_2\times V_3.\]
\end{example}

\begin{example}
Let $k=4$. 
There is a bijection 
\begin{equation}
\label{eqn_nil_quadruples_vs}
\begin{aligned}
\prod_{i=1}^{4} V_i \times \mcN(\beta) 
 &\simeq \Rep(Q,\beta) \times 
  \bigcup_{j=1}^{4} \Big(\prod_{i=1}^{j-1} V_{i} \times\{0\}\times\prod_{i=j+1}^4V_i\Big), \\
&=\Rep(Q,\beta) \times \{(v_1,v_2,v_3,v_4)\in\prod_{i=1}^4V_i : \text{some }v_i=0\}.
\end{aligned}
\end{equation}
Over $\F_q$, 
taking cardinalities in \eqref{eqn_nil_quadruples_vs} gives  
\begin{equation}
\label{card_nilp_quadruples}
\begin{split}
q^{\beta_1 + \beta_2 + \beta_3 + \beta_4}|&\mcN(\beta)| =  q^{\beta_1 \beta_2 + \beta_2 \beta_3 + \beta_3 \beta_4 + \beta_4 \beta_1} \Big(q^{\beta_1 + \beta_2 +\beta_3} + q^{\beta_2 + \beta_3 + \beta_4}  \\ 
& + q^{\beta_3 + \beta_4 + \beta_1} + q^{\beta_4 + \beta_1 + \beta_2} 
- (q^{\beta_1 + \beta_2} + q^{\beta_1 + \beta_3} + q^{\beta_1 + \beta_4} + q^{\beta_2 + \beta_3}  \\
& + q^{\beta_2 + \beta_4} + q^{\beta_3 + \beta_4})  
+ (q^{\beta_1} + q^{\beta_2} + q^{\beta_3} + q^{\beta_4}) - 1\Big). \\ 
\end{split}
\end{equation}
The large parenthesis on the right hand side of \eqref{card_nilp_quadruples} can be interpreted as all points minus those points whose all four coordinates are nonzero:
\[
q^{\beta_1 + \beta_2 + \beta_3 + \beta_4} - (q^{\beta_1}-1)(q^{\beta_2}-1)(q^{\beta_3}-1)(q^{\beta_4}-1).
\]
\end{example}

\subsection{Probability of nilpotent cyclic representations}
\label{subsection_prob_cyclic_quiver_field}
Let $\kk=\mathbb F_q$, where $q$ is a power of a prime.

\begin{proposition}
\label{prop_prob_nil_rep_fin_field}
The probability that a representation in $\Rep(Q,\beta)$ is nilpotent is
\begin{equation}
\label{eqn_prob_nil_cyclic_rep}
\Prob(\mcN(\beta)) = 
\frac{|Z|}{\prod_{i=1}^k| V_i| } 
= 1 - \prod_{i=1}^k(1 - q^{-\beta_i}).
\end{equation}
\end{proposition}

\begin{proof}
Dividing the number $ N(\beta) $ of nilpotent cyclic representations in~\eqref{eqn_cardinality_formula} by the total number of cyclic representations, $|\Rep(Q,\beta)| = q^{\beta\cdot\sigma\beta}$, gives the result.
\end{proof}

\begin{corollary}
\label{cor_limit_nil_cyclic_rep}
We have 
\begin{equation}
\label{eqn_prob_limit_finite_field}
0\leq \Prob(\mcN(\beta)) \leq \sum_{i=1}^{k} q^{-\beta_i}.
\end{equation}
Therefore $\Prob(\mcN(\beta))\rightarrow 0$ as $\sum_{i=1}^{k} q^{-\beta_i}\rightarrow 0$.
\end{corollary}

\begin{proof}
Since $q\geq 2$ and $\beta_i\geq 0$ for each $1\leq i\leq k$,
we have $0\leq q^{-\beta_i}\leq 1$ for each $i$, so the Bonferroni inequality~\cite{Gri20,Sta97,GKP94,Bru10} gives
$1 - \prod_{i=1}^k(1-q^{-\beta_i})\leq \sum_{i=1}^{k}q^{-\beta_i}$.
\end{proof}

\section{Nilpotent semirepresentations of cyclic quivers over the Boolean semiring}
\label{section_nilp_rep_Boolean}

In this section, we generalize~\cite[Section 4]{CIKLR25} to cyclic equioriented quivers. We use the notation for quiver representations given in Section~\ref{subsection_rep_quivers} and the graph constructions established in Section~\ref{subsection_directed_walks_graphs}.

\subsection{Recursion for nilpotent semirepresentations}
\label{subsection_nilp_semirepresentations}

For each dimension vector $\beta\in\mathbb Z_{\geq 0}^k$, take free $\mathbb B$-semimodules $M_1,\ldots,M_k$ with $\rank \: M_i=\beta_i$. Fix an ordered basis for each semimodule $M_i$.

Define the directed graph $G_\beta$ with vertex and edge sets given by
\[V(G_\beta)=\coprod_{i=1}^k[\beta_i],\quad E(G_\beta)=\coprod_{i=1}^k\left([\beta_i]\times[\beta_{i+1}]\right).\]
Vertices in $[\beta_i]$ have an edge directed to each of the vertices in $[\beta_{i+1}]$. Note that 
\[|V(G_\beta)|=\sum_{i=1}^k\beta_i=\mathbf 1\cdot\beta,\quad|E(G_\beta)|=\sum_{i=1}^k\beta_i\beta_{i+1}=\beta\cdot\sigma\beta.\]

For each $f\in\Rep_\mathbb B(Q,\beta)$ and $1\leq i\leq k$, let $F^i$ be the matrix of $f_i:M_i\to M_{i+1}$ with respect to the ordered bases of $M_i$, $M_{i+1}$.
Then, let $G(f)$ be the spanning subgraph of $G_\beta$ that contains the edge $(a,b)\in[\beta_i]\times[\beta_{i+1}]$ precisely when $F^i_{b,a}=1$.

The assignment $f\mapsto G(f)$ bijects $\Rep_\mathbb B(Q,\beta)$ to the set of spanning subgraphs of $G_\beta$; the representations $f\in\Rep_\mathbb B(Q,\beta)$ are $k$-tuples $(f_1,\ldots,f_k)$ with $f_i\in \Hom_{\mathbb B}(M_i,M_{i+1})$, and each set $\Hom_{\mathbb B}(M_i,M_{i+1})$ identifies with the $\beta_{i+1}\times\beta_i$ matrices over $\mathbb B$, which
correspond to all possible collections of edges from $[\beta_i]$ to $[\beta_{i+1}]$.
We claim that this map restricts to a bijection from $\mathcal N_\mathbb B(\beta)$ to $\DAG(G_\beta)$.
For each $f\in\Rep_\B(Q,\beta)$, $1\leq i\leq k$, $\ell\geq 1$, and $a\in[\beta_i]$, $b\in[\beta_{i+\ell}]$,
\begin{equation}\label{eqn_iterated_composition_computation}
    (F^{i+\ell-1}\cdots F^i)_{b,a} = \sum_{u_{\ell-1}=1}^{\beta_{i+\ell-1}}
    \ldots \sum_{u_{1}=1}^{\beta_{i+1}}F^i_{u_1,a}F^{i+1}_{u_2,u_1}\cdots F^{i+\ell-1}_{b,u_{\ell-1}}.
\end{equation}
The matrix entry~\eqref{eqn_iterated_composition_computation} is $0$ if and only if there is no directed walk
in $G(f)$ of length $\ell$ from $a$ to $b$. Therefore $f$ is nilpotent if and only if $G(f)$ is acyclic, since
all sufficiently long compositions $f_{i+\ell-1}\circ\cdots\circ f_i$ are zero precisely if walks in $G(f)$ are bounded in length.
By this correspondence, we have
\begin{equation}\label{eqn_DAG_nilpotent_correspondence}
    N_{\B}(\beta)=|\DAG(G_\beta)|.
\end{equation}

For each $U\subseteq V(G_\beta)$, define subsets $U_i:=U\cap[\beta_i]$ for $1\leq i\leq k$ and let $\mathbf U\in\mathbb Z_{\geq 0}^k$ be the dimension vector $\mathbf U:=(|U_1|,\ldots,|U_k|)$. The induced graph $G_\beta[U]$ is isomorphic to the graph $G_{\mathbf U}$ defined by the dimension vector $\mathbf U$, so~\eqref{eqn_DAG_nilpotent_correspondence} gives
\[|\DAG(G_\beta[U])|=|\DAG(G_\mathbf U)|=N_{\B}(\mathbf U).\]
Also, write $U^{\c}$ for $V(G_\beta)\setminus U$ and let $\mathcal S(U):=\{D\in\DAG(G_\beta) : U\subseteq\Src(D)\}$.

\begin{lemma}\label{lmm_counting_S(U)}
    For each nonempty vertex subset $U\subseteq V(G_\beta)$, 
    \[|\mathcal S(U)|=2^{\mathbf U\cdot\sigma(\beta-\mathbf U)} N_{\B}(\beta-\mathbf U).\]
\end{lemma}
\begin{proof}
Let $\mathcal A(U)$ denote the set of spanning subgraphs of $G_\beta$ whose edges have source in $U$ and target in $U^{\c}$.
There is a bijection
    \begin{equation}\label{eqn_S(U)_bijection}
        \DAG(G_\beta[U^{\c}])\times\mathcal A(U)\xrightarrow\simeq\mathcal S(U),\quad (D,D')\mapsto D\cup D'.
    \end{equation}

    Fix any pair $(D,D')$. All edges in both graphs $D$, $D'$ have target in $U^{\c}$, so $U\subseteq\Src(D\cup D')$. Also, $D\cup D'$ is acyclic since any closed walk would have to lie entirely in $U^{\c}$; such a walk would also be a walk of the acyclic graph $D$, which is impossible.
    Therefore $D\cup D'\in\mathcal S(U)$.

    The inverse map is given by taking $D\in\mathcal S(U)$ to the pair $(D[U^{\c}],D')$, where $D'$ is the spanning subgraph of $D$ with the edge set
    \[
    E(D'):=\{e\in E(D) : e\text{ has source in }U \}.
    \]
    The induced graph $D[U^{\c}]$ is acyclic as a subgraph of the acyclic graph $D$. Next, $D'\in\mathcal A(U)$ since
    each edge $e\in E(D')$ has source in $U$ and belongs to $E(D)$, so has target in $U^{\c}$.

 The bijection implies $|\mathcal S(U)|=|\mathcal A(U)|\cdot |\DAG(G_\beta[U^{\c}])|$.
 For each graph $D\in\mathcal A(U)$ and each $1\leq i\leq k$, $D$ has $\beta_{i+1}-|U_{i+1}|$ possible edges directed from each vertex $v\in U_i$ to the vertices in $[\beta_{i+1}]\setminus U_{i+1}$. Therefore,
    \[|\mathcal A(U)|=2^{\sum_{i=1}^k |U_i|(\beta_{i+1}-|U_{i+1}|)}=2^{\mathbf U\cdot\sigma(\beta-\mathbf U)}.\]
The result follows since $N_{\B}(\beta-\mathbf U)=|\DAG(G_\beta[U^{\c}])|$.
\end{proof}

\begin{theorem}
\label{thm_boolean_nilp_cardinality_recursion}
    For each nonzero dimension vector $\beta\in\mathbb Z_{\geq 0}^k$, the number of nilpotent semirepresentations of the cyclic quiver over $\B$ is
    \[
    N_{\B}(\beta) 
    = \sum_{\gamma<\beta}(-1)^{\mathbf 1\cdot(\beta-\gamma)-1}\binom{\beta}{\gamma}2^{(\beta-\gamma)\cdot\sigma\gamma}N_{\B}(\gamma).\]
\end{theorem}
\begin{proof}
    Each directed acyclic subgraph $D\in\DAG(G_\beta)$ has a source vertex $v\in V(G_\beta)$. Thus, the principle of inclusion-exclusion gives
    \[
    N_{\B}(\beta)=\sum_{\emptyset\neq U\subseteq V(G_\beta)}(-1)^{|U|-1}\left|\bigcap_{v\in U}\mathcal S(\{v\})\right|=\sum_{\emptyset\neq U\subseteq V(G_\beta)}(-1)^{|U|-1}\left|\mathcal S(U)\right|.\]
    For each $U\subseteq V(G_\beta)$, $|U^\c|=|V(G_\beta)\setminus U|=\sum_{i=1}^k|[\beta_i]\setminus U_i|=\mathbf 1\cdot(\beta-\mathbf U)$.
    Thus, reindexing the summation by $U\mapsto U^\c$ and then applying Lemma~\ref{lmm_counting_S(U)} for the cardinality of $\mathcal S(U^{\c})$ gives
    \begin{equation}\label{eqn_recursive_DAG_cardinality_subset_sum}
        \begin{aligned}
        N_{\B}(\beta) &= \sum_{U\subsetneq V(G_\beta)}(-1)^{|U^{\c}|-1}|\mathcal S(U^{\c})|\\
        &=\sum_{U\subsetneq V(G_\beta)}(-1)^{\mathbf 1\cdot(\beta-\mathbf U)-1}2^{(\beta-\mathbf U)\cdot\sigma\mathbf U}N_{\B}(\mathbf U).
    \end{aligned}
    \end{equation}
For each dimension vector $\gamma<\beta$, there are $\binom{\beta}{\gamma}=\prod_{i=1}^k\binom{\beta_i}{\gamma_i}$ subsets $U\subsetneq V(G_\beta)$ such that $\mathbf U=\gamma$; for each $1\leq i\leq k$, there are $\binom{\beta_i}{\gamma_i}$ subsets $U_i\subseteq[\beta_i]$ with $|U_i|=\gamma_i$.
Subsets with the same dimension vector contribute the same summand, so reindexing~\eqref{eqn_recursive_DAG_cardinality_subset_sum} by dimension vectors yields the result,
\[N_{\B}(\beta) =\sum_{\gamma<\beta}\sum_{\mathbf U=\gamma}(-1)^{\mathbf 1\cdot(\beta-\mathbf U)-1}2^{(\beta-\mathbf U)\cdot\sigma\mathbf U}N_{\B}(\mathbf U)=\sum_{\gamma<\beta}(-1)^{\mathbf 1\cdot(\beta-\gamma)-1}\binom{\beta}{\gamma}2^{(\beta-\gamma)\cdot\sigma\gamma}N_{\B}(\gamma).\qedhere\]
\end{proof}

\begin{corollary}
    If 
    $\beta = (1,1,\ldots, 1)$, then 
    $N_{\B}(\beta)=2^k-1$.
\end{corollary}

\begin{proof}
A semirepresentation $f\in \Rep_{\B}(Q,\beta)$ is nilpotent if and only if at least one of the $f_i=0$. This means all semirepresentations are nilpotent except one, $f=(1,1,\ldots, 1)$.
\end{proof}

The dimension vector $\beta = (1,1,\ldots, 1)$ gives the simplest possible Boolean case since there is only one edge between each consecutive pair of vertices. So the only non-nilpotent semirepresentation is the directed $k$-cycle.

\subsection{Probability of nilpotent cyclic tuples}
\label{subsection_prob_Boolean}
We continue to work over $\B$.
For each dimension vector $\beta\in\mathbb Z_{\geq 0}^k$, let $m(\beta):=\min_{1\leq i\leq k}\beta_i\beta_{i+1}$.
If any $\beta_i=0$, then all semirepresentations are trivially nilpotent. Thus suppose $\beta\in\mathbb Z_{\geq 1}^k$, which implies $m(\beta)>0$.

Let $C_\beta$ be the set of $k$-cycles in $G_\beta$, which are the closed walks
\[(v_1,v_2),(v_2,v_3),\ldots,(v_k,v_1),\quad v_i\in[\beta_i].\]

\begin{lemma}\label{lmm_fiber_bound}
    Every edge $(u,v)\in E(G_\beta)$ is contained in at most $\frac{\prod_{i=1}^k\beta_i}{m(\beta)}$ cycles in $C_\beta$.
\end{lemma}
\begin{proof}
Let $u\in[\beta_j]$, $v\in[\beta_{j+1}]$. Then $(u,v)$ is contained only in $k$-cycles of the form
\[(v_1,v_2),(v_2,v_3),\ldots,(v_{j-1},u),(u,v),(v,v_{j+2}),(v_{j+2},v_{j+3}),\ldots,(v_k,v_1).\]
Varying the vertices $v_1,\ldots,v_{j-1},v_{j+2},\ldots,v_k$ with each $v_i\in[\beta_i]$ gives $\prod_{i\neq j,j+1}\beta_i\leq\frac{\prod_{i=1}^k\beta_i}{m(\beta)}$ possible $k$-cycles that contain $(u,v)$.
\end{proof}

\begin{lemma}
\label{lemma_nil_rep_cyclic_Boolean}
For all dimension vectors $\beta\in\mathbb Z_{\geq 1}^k$, the cardinality $N_{\B}(\beta)$ of nilpotent semirepresentations over $\B$ is bounded by 
\begin{equation}
\label{eqn_card_est_nil_rep_cyclic_boolean}
    2^{\beta\cdot\sigma\beta-m(\beta)}
    \leq
    N_{\B}(\beta)
    \leq \min\left((\mathbf 1\cdot\beta)!2^{\beta\cdot\sigma\beta-m(\beta)},(1-2^{-k})^{m(\beta)/k}2^{\beta\cdot\sigma\beta}\right).
\end{equation}
Therefore, the probability that a semirepresentation in $\Rep_{\B}(Q,\beta)$ is nilpotent is bounded by
\begin{equation}
\label{eqn_prob_nil_rep_cyclic_Boolean}
2^{-m(\beta)}\leq\Prob(\mcN_{\B}(\beta))\leq \min\left((\mathbf 1\cdot\beta)!2^{-m(\beta)},(1-2^{-k})^{m(\beta)/k}\right).
\end{equation}
\end{lemma}

\begin{proof}
For each binary relation $U\subseteq V(G_\beta)^2$, let $\mathcal E(U)$ be the set of spanning subgraphs $D$ of $G_\beta$ with $E(D)\subseteq U$. We view orders on $V(G_\beta)$ and subsets $U\subseteq E(G_\beta)$ as binary relations of $V(G_\beta)$, and let $T$ denote the set of strict total orders on $V(G_\beta)$. Fix $j\in[k]$ so that $\beta_j\beta_{j+1}=m(\beta)$.

For the lower bound, define $U_j:=E(G_\beta)\setminus ([\beta_j]\times[\beta_{j+1}])$.
Any cycle in $G_\beta$ must pass from $[\beta_j]$ to $[\beta_{j+1}]$, so all graphs $D\in\mathcal E(U_j)$ are acyclic. Therefore $|\mathcal E(U_j)|\leq N_{\B}(\beta)$. Since $|U_j|=\beta\cdot\sigma\beta-m(\beta)$, we have
\[
2^{\beta\cdot\sigma\beta-m(\beta)}=|\mathcal E(U_j)|\leq N_{\B}(\beta).
\]

To establish $N_\mathbb B(\beta)\leq (\mathbf 1\cdot\beta)!2^{\beta\cdot\sigma\beta-m(\beta)}$, we show $\DAG(G_\beta)\subseteq\bigcup_{R\in T}\mathcal E(R)$. This implies
\begin{equation}\label{eqn_card_est_nil_rep_upper_bound_E(R)_covering}
N_{\B}(\beta) \leq \sum_{R\in T}|\mathcal E(R)|.
\end{equation}
Fix any $D\in\DAG(G_\beta)$ and an order $<_0\in T$. The \defn{height} of a vertex $u\in V(G_\beta)$ is the maximum length of a walk in $D$ with target $u$. 
Then, define $<_D\in T$ by declaring $u<_Dv$ if either $u$ has lesser height than $v$, or if $u$ and $v$ are of the same height and $u <_0v$.
For each edge $(u,v)\in E(D)$, $u$ is of lesser height than $v$, so $u<_Dv$. Therefore $D\in\mathcal E(<_D)$, so $\DAG(G_\beta)\subseteq\bigcup_{R\in T}\mathcal E(R)$.

Next, fix any $R\in T$. To bound $|\mathcal E(R)|$, write $R^{\c}:=E(G_\beta)\setminus R$ and define a map $\varphi:C_\beta\to R^{\c}$ that sends each $k$-cycle $C\in C_\beta$ 
to any edge $(u,v)\in E(C)\setminus R$. Such an edge $(u,v)$ exists, otherwise $R$ contains the cycle $C$. By Lemma~\ref{lmm_fiber_bound}, the size of each fiber $\varphi^{-1}((u,v))$ is bounded by $\frac{\prod_{i=1}^k\beta_i}{m(\beta)}$.
There are $|R^\c|$ such fibers, and these fibers partition the domain $C_\beta$. Since $C_\beta$ has cardinality $\prod_{i=1}^k\beta_i$, we have
\[\prod_{i=1}^k\beta_i\leq |R^{\c}|\frac{\prod_{i=1}^k\beta_i}{m(\beta)}.\]
Therefore $m(\beta)\leq |R^{\c}|$.
Graphs $D\in\mathcal E(R)$ have $\beta\cdot\sigma\beta-|R^{\c}|$ possible edges in $R\cap E(G_\beta)$, and $\beta\cdot\sigma\beta-|R^{\c}|\leq \beta\cdot\sigma\beta-m(\beta)$. Therefore,
\begin{equation}\label{eqn_card_est_nil_rep_upper_bound_E(R)_size}
    |\mathcal E(R)|\leq2^{\beta\cdot\sigma\beta-|R^{\c}|}\leq 2^{\beta\cdot\sigma\beta-m(\beta)}.
\end{equation}
Since $|V(G_\beta)|=\mathbf 1\cdot\beta$, there are $(\mathbf 1\cdot\beta)!$ strict total orders on $V(G_\beta)$. Thus~\eqref{eqn_card_est_nil_rep_upper_bound_E(R)_covering} and~\eqref{eqn_card_est_nil_rep_upper_bound_E(R)_size} give
\[N_{\B}(\beta)\leq \sum_{R\in T}|\mathcal E(R)|\leq (\mathbf 1\cdot\beta)!2^{\beta\cdot\sigma\beta-m(\beta)}.\]

Finally, we establish $N_{\B}(\beta)
\leq(1-2^{-k})^{m(\beta)/k}2^{\beta\cdot\sigma\beta}$. Each cycle  $C\in C_\beta$ is comprised of $k$ edges, and by Lemma~\ref{lmm_fiber_bound} each edge is contained in at most $\frac{\prod_{i=1}^k\beta_i}{m(\beta)}$ cycles in $C_\beta$. Therefore $C$ shares an edge with at most $k\frac{\prod_{i=1}^k\beta_i}{m(\beta)}$ cycles in $C_\beta$. Since $|C_\beta|=\prod_{i=1}^k\beta_i$, we can successively choose cycles $\{C_1,\ldots,C_n\}\subseteq C_\beta$ with $n\geq \frac{m(\beta)}{k}$, where each edge set $E(C_{i})$ is disjoint from $\bigcup_{j<i}E(C_j)$. Write $E:=\bigcup_{i=1}^nE(C_i)$.

For each $1\leq i\leq n$, there are $2^k-1$ subgraphs $D\in\mathcal E(E(C_i))$ that omit at least one edge of $C_i$. Thus, there are $(2^k-1)^n$ subgraphs $D\in \mathcal E(E)$ that omit at least one edge in each cycle $C_1,\ldots,C_n$. Since $|E(G_\beta)\setminus E|=\beta\cdot\sigma\beta-nk$, each $D\in\mathcal E(E)$ extends to
$2^{\beta\cdot\sigma\beta-nk}$ spanning subgraphs of $G_\beta$. Therefore, $(2^k-1)^n2^{\beta\cdot\sigma\beta-nk}$ spanning subgraphs omit at least one edge in each cycle. All DAGs $D\in\DAG(G_\beta)$ omit an edge from each cycle, which yields the bound
\[N_\mathbb B(\beta)\leq(2^k-1)^n2^{\beta\cdot\sigma\beta-nk}\leq(1-2^{-k})^{m(\beta)/k}2^{\beta\cdot\sigma\beta}.\qedhere \]
\end{proof}

\begin{corollary}
If $\beta^{(1)},\beta^{(2)},\ldots$ is a sequence of dimension vectors and $m(\beta^{(n)})\to\infty$, then
\[\Prob(\mathcal N_\B(\beta^{(n)}))\to 0.\]
\end{corollary}
\begin{proof}
    This follows from~\eqref{eqn_prob_nil_rep_cyclic_Boolean}, since for each $n\in\mathbb N$,
$\Prob(\mcN_{\B}(\beta^{(n)}))\leq(1-2^{-k})^{m(\beta^{(n)})/k}$.
\end{proof}

For any dimension vector $\beta\in\mathbb Z_{\geq 1}^k$ and $n\in\mathbb N$, let $n\beta:=(n\beta_1,\ldots,n\beta_k)$.

\begin{corollary}
Fix any dimension vector $\beta\in\mathbb Z_{\geq 1}^k$. Then
\[\lim_{n\to\infty}\left(\frac{\log_2\left(\Prob(\mathcal N_\B(n\beta)\right)}{n^2}\right)=-m(\beta).\]
\end{corollary}
\begin{proof}
    For each $n\in\mathbb N$, $m(n\beta)=n^2m(\beta)$ and $\mathbf 1\cdot (n\beta)=n(\mathbf 1\cdot\beta)$. Taking logarithms in~\eqref{eqn_prob_nil_rep_cyclic_Boolean} and dividing by $n^2$ gives
    \[-m(\beta)\leq\frac{\log_2\left(\Prob(\mcN_{\B}(n\beta))\right)}{n^2}\leq \frac{\log_2((n(\mathbf 1\cdot\beta))!)}{n^2}-m(\beta).\]
    The result follows since $\frac{\log_2((n(\mathbf 1\cdot\beta))!)}{n^2}\to 0$. To see this, note $(n(\mathbf 1\cdot\beta))!\leq (n(\mathbf 1\cdot\beta))^{n(\mathbf 1\cdot\beta)}$. Therefore
    \[\frac{\log_2((n(\mathbf 1\cdot\beta))!)}{n^2}\leq \frac{(\mathbf 1\cdot\beta)\log_2(n(\mathbf 1\cdot\beta))}{n}\to 0.\qedhere\]
\end{proof}

\section{Comparing eventually constant representations to nilpotent representations}
\label{section_comparing_set_theoretic_vs_eqns}

Let $Q$ be the cyclic quiver with path category $\mathcal Q$, and suppose $\kk$ is the finite field $\mathbb F_q$. We compare representations of $Q$ over $\mathbb F_q$, functors from $\mathcal Q$ to $\Vect_{\mathbb F_q}$, to the \defn{set-valued representations} of $Q$, functors from $\mathcal Q$ to the category $\FinSet$ of finite sets. Let $\Rep_\Set(Q)$ denote the functor category from $\mathcal Q$ to $\FinSet$, and $\Rep(Q)$ the functor category from $\mathcal Q$ to $\Vect_{\mathbb F_q}$.
A set-valued representation $W$ is \defn{eventually constant} if there exists a $k$-tuple $(c_1,\ldots,c_k)\in\prod_{i=1}^kW_i$ such that for all sufficiently long paths $p:i\to j$ in $\mathcal Q$, $W_p$ maps $W_i$ to $\{c_j\}\subseteq W_j$.

In the set-theoretic context, an element $\beta\in\mathbb Z_{\geq 1}^k$ is a \defn{cardinality vector}. For fixed sets $X_1,\ldots,X_k$ of size $|X_i|=\beta_i$, $\Rep_\Set(Q,\beta)\subseteq\Rep_\Set(Q)$ denotes the set-valued representations taking each vertex $i\in[k]$ to the set $X_i$. As in the linear case, $\Rep_\Set(Q,\beta)$ identifies with a product, $\prod_{i=1}^k\Hom_\Set(X_i,X_{i+1})$. Let $\mathcal E\mathcal C(Q,\beta)\subseteq\Rep_\Set(Q,\beta)$ be the subset of eventually constant representations and note that $f\in\Rep_\Set(Q,\beta)$ is eventually constant if and only if the composition $f_k\cdots f_1$ is eventually constant as an endomorphism of $X_1$.

First, we consider functors that take linear representations to set-valued representations. 
Fix a dimension vector $\beta\in\mathbb Z_{\geq 0}^k$ with vector spaces $V_1,\ldots,V_k$ of dimension $\dim V_i=\beta_i$, and define the cardinality vector $q^\beta:=(q^{\beta_1},\ldots,q^{\beta_k})\in\mathbb Z_{\geq 1}^k$.
Let $U:\Vect_\kk\to\FinSet$ be the forgetful functor that takes each linear map $f:V\to V'$ to its underlying function on sets, $U(f):U(V)\to U(V')$.
The functor $U$ induces a functor $U_*:\Rep(Q)\to\Rep_\Set(Q)$ by
post-composition; $U_*$ sends each linear representation $W:\mathcal Q\to\Vect_{\mathbb F_q}$ to the set-valued representation
$U\circ W:\mathcal Q\to\FinSet$.

Since each $|U(V_i)|=q^{\beta_i}$, we can take the representation spaces
\[\Rep(Q,\beta)=\prod_{i=1}^k\Hom(V_i,V_{i+1}),\quad\Rep_\Set(Q,q^\beta)=\prod_{i=1}^k\Hom_\Set(U(V_i),U(V_{i+1})).\]
Concretely, $U_*$ takes each $f\in\Rep(Q,\beta)$ to
$(U(f_1),\ldots,U(f_k))\in\Rep_\Set(Q,q^\beta)$.

\begin{proposition}
    For each linear representation $f\in\Rep(Q,\beta)$, $f\in\mathcal N(\beta)$ if and only if $U_*(f)\in\mathcal E\mathcal C(Q,q^\beta)$.
\end{proposition}
\begin{proof}
    The linear representation $f$ is nilpotent if and only if $(f_k\cdots f_1)^n=0$ for sufficiently large $n\in\mathbb N$. Likewise, $U_*(f)$ is eventually constant if and only if $(U(f_k)\cdots U(f_1))^n$ is constant for sufficiently large $n\in\mathbb N$. The result follows since the linear map $(f_k\cdots f_1)^n$ is zero if and only if it is constant and $(U(f_k)\cdots U(f_1))^n=U((f_k\cdots f_1)^n)$ by functoriality.
\end{proof}

For each representation $f\in\Rep(Q,\beta)$, each linear map $f_i$ is determined by its underlying map $U(f_i)$. Thus $U_*$ is injective on $\Rep(Q,\beta)$ and yields a bijection
\[\mathcal N(\beta)\xrightarrow\cong\mathcal E\mathcal C(Q,q^\beta)\cap U_*(\Rep(Q,\beta)).\]
Since $|\Rep(Q,\beta)|=|U_*(\Rep(Q,\beta))|$, we have
\[\frac{N(\beta)}{|\Rep(Q,\beta)|}=\frac{|\mathcal E\mathcal C(Q,q^\beta)\cap U_*(\Rep(Q,\beta))|}{|U_*(\Rep(Q,\beta))|}.\]
\begin{proposition}
The probability that a linear representation $f\in\Rep(Q,\beta)$ is nilpotent is equal to the probability that a set-valued representation $f\in\Rep_\Set(Q,q^\beta)$ is eventually constant,
    \begin{equation*}
    \frac{N(\beta)}{|\Rep(Q,\beta)|} 
    = 1 - \prod_{i=1}^k(1-q^{-\beta_i})
    = \frac{|\mathcal E\mathcal C(Q,q^\beta)|}{|\Rep_\Set(Q,q^\beta)|}.
    \end{equation*}
\end{proposition}
\begin{proof}
The first equality follows from~\eqref{eqn_prob_nil_cyclic_rep}.
For the second equality, 
we have 
\begin{equation*}
|\mathcal E\mathcal C(Q,q^\beta)| = q^{q^\beta\cdot\sigma\beta-\mathbf 1\cdot\beta}\left( q^{\mathbf 1\cdot\beta} 
- \prod_{j=1}^{k} (q^{\beta_j} -1) \right), 
\end{equation*}
by~\cite[Proposition 3.5]{GHI26}, and
\[|\Rep_\Set(Q,q^\beta)| = \left|\prod_{i=1}^k\Hom_\Set(U(V_i),U(V_{i+1}))\right| = q^{\sum_{i=1}^{k} \beta_{i+1}q^{\beta_i}}=q^{q^\beta\cdot\sigma\beta}.\]
We thus obtain that 
$|\mathcal E\mathcal C(Q,q^\beta)|/|\Rep_\Set(Q,q^\beta)| = 1 - \prod_{i=1}^k(1-q^{-\beta_i})$. 
\end{proof}

Next, we consider a functor that takes set-valued representations to linear representations. Let $\kk$ be any field and fix a cardinality vector $\beta\in\mathbb Z_{\geq 1}^k$ with sets $X_1,\ldots,X_k$ satisfying $|X_i|=\beta_i$.

The \defn{linearization functor}
$\kk[-]:\FinSet\to\Vect_\kk$ takes each set $X$ to the vector space generated by $X$,
$\kk[X] := \left\{\sum_{x\in X}k_x\cdot x : k_x\in\kk \right\}$. It sends each function $f:X\to Y$ to the linear map
\begin{equation}
    \kk[f]:\kk[X]\to\kk[Y],\quad\sum_{x\in X}k_x\cdot x\mapsto\sum_{x\in X}k_x\cdot f(x).
\end{equation}
Write $[1]=\{1\}$ and identify $\kk[[1]]$ with $\kk$. For each set $X$, write $!_X$ for the unique map $!_X:X\to[1]$.
All functions $f:X\to Y$ satisfy $!_Y\circ f=!_X$, so functoriality of $\kk[-]$ gives
\begin{equation}\label{eqn_naturality_kk[-]}
\kk[!_Y]\circ\kk[f]=\kk[!_X].
\end{equation}

For all sets $X\in\FinSet$, define 
\[
\widetilde H_0(X,\kk):=\ker\!(\kk[!_X]) 
= \left\{ \sum_{x\in X}k_x\cdot x\in\kk[X] : \sum_{x\in X}k_x=0 \right\} \subseteq \kk[X].
\] By~\eqref{eqn_naturality_kk[-]}, $\kk[f]$ maps $\widetilde H_0(X,\kk)$ to $\widetilde H_0(Y,\kk)$, so let $\widetilde H_0(f,\kk):\widetilde H_0(X,\kk)\to\widetilde H_0(Y,\kk)$ be the restriction of $\kk[f]$. This gives the \defn{reduced zeroth homology functor} $\widetilde H_0(-,\kk):\FinSet\to\Vect_\kk$.
For nonempty $X$, rank-nullity gives $\dim\widetilde H_0(X,\kk)=\dim\kk[X]-1=|X|-1$.

\begin{proposition}\label{prop_constant_0}
    For any function $f:X\to Y$, $f$ is constant if and only if $\widetilde H_0(f,\kk)=0$.
\end{proposition}
\begin{proof}
    If $f$ is constant at $y_0\in Y$, then for all $\sum_{x\in X}k_x\cdot x\in\widetilde H_0(X,\kk)$,
    \[\widetilde H_0(f,\kk)\left(\sum_{x\in X}k_x\cdot x\right)=\sum_{x\in X}k_x\cdot f(x)=\left(\sum_{x\in X}k_x\right)\cdot y_0=0.\]
    Conversely, if $\widetilde H_0(f,\kk)=0$, then for all distinct $x,\,x'\in X$, $x-x'\in \widetilde H_0(X,\kk)$ and
    \[f(x)-f(x')=\widetilde H_0(f,\kk)\left(x-x'\right)=0,\]
    so $f(x)=f(x')$. As this holds for all $x,\,x'\in X$, $f$ is constant.
\end{proof}

The functor $\widetilde H_0(-,\kk):\FinSet\to\Vect_\kk$ induces a functor
$\widetilde H_0(-,\kk)_*:\Rep_\Set(Q)\to\Rep(Q)$ by post-composition. Concretely, $f\in\Rep_\Set(Q,\beta)$ maps to $(\widetilde H_0(f_1,\kk),\ldots,\widetilde H_0(f_k,\kk))\in\Rep(Q)$, with dimension vector $\beta-\mathbf 1$.

\begin{proposition}\label{prop_eventually_constant_nilpotent}
A set-valued representation $f\in\Rep_\Set(Q,\beta)$ is eventually constant if and only if $\widetilde H_0(-,\kk)_*(f)$ is nilpotent.
\end{proposition}
\begin{proof}
    The set-valued representation $f$ is eventually constant if and only if $(f_k\cdots f_1)^n:X_1\to X_1$ is constant for sufficiently large $n\in\mathbb N$. Likewise, the linear representation
    $\widetilde H_0(-,\kk)_*(f)$ is nilpotent if and only if $(\widetilde H_0(f_k,\kk)\cdots\widetilde H_0(f_1,\kk))^n$ is $0$ for sufficiently large $n\in\mathbb N$. By functoriality,
    \[(\widetilde H_0(f_k,\kk)\cdots\widetilde H_0(f_1,\kk))^n=\widetilde H_0((f_k\cdots f_1)^n,\kk).\]
    Therefore these conditions are equivalent by Proposition~\ref{prop_constant_0}.
\end{proof}


\bibliographystyle{amsalpha} 
\bibliography{nilpotent_finite}

\end{document}